\documentclass[11pt]{amsart}
\usepackage{amsthm}

\theoremstyle{plain}
\newtheorem{thm}{Theorem}[section]
\newtheorem{theorem}[thm]{Theorem}

\theoremstyle{definition}

\numberwithin{equation}{section}

 \title{Asymptotic behaviour of learning rates in Armijo's condition}
 \author{Tuyen Trung Truong}
   \address{Department of Mathematics, University of Oslo, Blindern 0851 Oslo, Norway}
  \email{tuyentt@math.uio.no}
 \author{Tuan Hang Nguyen}
\address{Axon AI Research}
\email{hnguyen@axon.com}
    \date{\today}
   \subjclass[2010]{}

\begin{document}
\maketitle

\begin{abstract}

Fix a constant $0<\alpha <1$. For a $C^1$ function $f:\mathbb{R}^k\rightarrow \mathbb{R}$, a point $x$ and a positive number $\delta >0$, we say that Armijo's condition is satisfied if $f(x-\delta \nabla f(x))-f(x)\leq -\alpha \delta ||\nabla f(x)||^2$. It is a basis for the well known Backtracking Gradient Descent (Backtracking GD) algorithm. 

Consider a sequence $\{x_n\}$ defined by $x_{n+1}=x_n-\delta _n\nabla f(x_n)$, for positive numbers $\delta _n$ for which Armijo's condition is satisfied. We show that if $\{x_n\}$ converges to a non-degenerate critical point, then $\{\delta _n\}$ must be bounded. Moreover this boundedness can be quantified in terms of the norms of the Hessian $\nabla ^2f$ and its inverse at the limit point. This complements the first author's results on Unbounded Backtracking GD, and shows that in case of convergence to a non-degenerate critical point the behaviour of Unbounded Backtracking GD is not too different from that of usual Backtracking GD. On the other hand, in case of convergence to a degenerate critical point the behaviours can be very much different. We run some experiments to illustrate that both scenrios can really happen.  

In another part of the paper, we argue that Backtracking GD has the correct unit (according to a definition by Zeiler in his Adadelta's paper). The main point is that since learning rate in Backtracking GD is bound by Armijo's condition, it is not unitless.      

\end{abstract}


\subsection{Asymptotic behaviour of learning rates in Armijo's condition}

Fix a constant $0<\alpha <1$. For a $C^1$ function $f:\mathbb{R}^k\rightarrow \mathbb{R}$, a point $x$ and a positive number $\delta >0$, we say that Armijo's condition is satisfied if $f(x-\delta \nabla f(x))-f(x)\leq -\alpha \delta ||\nabla f(x)||^2$. We say that a sequence $\{x_n\}$ satisfies Armijo's condition \cite{armijo} if $x_{n+1}=x_n-\delta _n\nabla f(x_n)$ for some positive number $\delta _n$ for which Armijo's condition is satisfied. 

In Backtracking GD, one fixes a countable set $\Delta$ of positive numbers converging to $0$, starts from a random  initial point  $x_0$ and defines $x_{n+1}=x_n-\delta (x_n)\nabla f(x_n)$, where $\delta (x_n)\in \Delta $ is the largest number for which Armijo's condition is satisfied. Convergence guarantee for Backtracking GD and modifications is currently the best among all iterative methods \cite{truong-nguyen} with associated python codes for experiments on CIFAR datasets \cite{mbtoptimizer}. A popular choice for the set $\Delta $ is as follows: we choose $0<\beta <1$ and $\delta _0>0$ and define $\Delta =\{\beta ^n\delta _0:~n=0,1,2,\ldots \}$

A drawback in Backtracking GD is that the learning rates are bounded from above by $\max \Delta $. If one could allow learning rates in Backtracking GD to be unbounded, then the convergence could be faster and could avoid bad critical points. To this end, the first author defined in \cite{truong} the Unbounded Backtracking GD procedure, where now learning rates $\delta _n$ are not bounded by $\max \Delta$ but are allowed to grow provided $\lim _{n\rightarrow\infty}\delta _n||\nabla f(x_n)||=0$. Under this condition, one obtains the same convergence guarantee as in Backtracking GD. 

If the sequence $\{x_n\}$ satisfies Armijo's condition and converges, then the above condition $\lim _{n\rightarrow\infty}\delta _n||\nabla f(x_n)||=0$ is satisfied. 
The goal of numerical optimisation is to guarantee convergence to local minima, and hence at least to critical points of $f$. 

Recall that a critical point of $f$ is non-degenerate at a critical point $x_{\infty}$ if $f$ is $C^2$ near $x_{\infty}$ and the Hessian $\nabla ^2f(x_{\infty})$ is invertible. Note that non-degenerate critical points are "generic", in the sense that a randomly chosen function $f$ will have all its critical points to be non-degenerate (for example, Morse's functions). The above discussion motivates us to investigate the question: Can we allow the sequence $\delta _n$ grow to infinity while having the sequence $\{x_n\}$ converge to a non-degenerate critical point? A bit surprisingly, the answer is No, as seen from the next result. 

\begin{theorem} Assume that the sequence $\{x_n\}$ satisfies Armijo's condition and converges to a non-degenerate critical point $x_{\infty}$. To avoid triviality, we assume moreover that $\nabla f(x_n)\not= 0$ for all $n$.  Then for every $\epsilon >0$, there is $n_{\epsilon} $ so that for all $n\geq n_{\epsilon}$ we have 
\begin{eqnarray*}
\alpha \delta _n\leq \frac{1}{2}(||\nabla ^2f(x_{\infty})||+\epsilon ) \times  (||\nabla ^2f(x_{\infty})^{-1}||+\epsilon )^2. 
\end{eqnarray*}
 \label{TheoremLearningRateRestriction}\end{theorem}
\begin{proof}
Fix $\epsilon >0$. We have that $\{f(x_n)\}$ decreases to $f(x_{\infty})$. Hence, by Armijo's condition we have
\begin{eqnarray*}
0\leq f(x_{n+1})-f(x_{\infty})\leq f(x_n)-f(x_{\infty}) -\alpha \delta _n||\nabla f(x_n)||^2,
\end{eqnarray*}
for all $n$. Therefore, for all $n$, we have $\alpha \delta _n||\nabla f(x_n)||^2 \leq  f(x_n)-f(x_{\infty})$. 

By Taylor's expansion for $f$ near $x_{\infty}$, using that $f$ is $C^2$  and noting that $\nabla f(x_{\infty})=0$, we have (here $o(.)$ is the small-O notation)
\begin{eqnarray*}
f(x_n)-f(x_{\infty}) &=&\frac{1}{2}<\nabla ^2f(x_{\infty})(x_n-x_{\infty}),x_n-x_{\infty}>+o(||x_n-x_{\infty}||^2)\\
&\leq& \frac{1}{2}||\nabla ^2f(x_{\infty})||\times ||x_n-x_{\infty}||^2+o(||x_n-x_{\infty}||^2).  
\end{eqnarray*}
Hence, if $n$ is large enough then $f(x_n)-f(x_{\infty})\leq \frac{1}{2}(||\nabla ^2f(x_{\infty})||+\epsilon )\times ||x_n-x_{\infty}||^2$

By Taylor's expansion for $\nabla f$ near $x_{\infty}$, using  again that $f$ is $C^2$  and noting that $\nabla f(x_{\infty})=0$, we have
\begin{eqnarray*}
\nabla f(x_n)=\nabla ^2f(x_{\infty})(x_n-x_{\infty})+o(||x_n-x_{\infty}||). 
\end{eqnarray*}
Hence, multiplying both sides with $\nabla ^2f(x_{\infty})^{-1}$,  when $n$ is large enough, we get $||x_n-x_{\infty}||\leq (||\nabla ^2f(x_{\infty})^{-1}||+\epsilon ) ||\nabla f(x_n)||$. 

Putting together all the above estimates and cancelling the term $||\nabla f(x_n)||^2$ at the end, we obtain finally: 
\begin{eqnarray*}
\alpha \delta _n\leq \frac{1}{2}(||\nabla ^2f(x_{\infty})||+\epsilon ) \times  (||\nabla ^2f(x_{\infty})^{-1}||+\epsilon )^2,
\end{eqnarray*}
for large enough values of $n$, as wanted.
\end{proof}

This result says roughly that in case of convergence to a non-degenerate critical point, then the performance of Unbounded Backtracking GD and of the usual Backtracking GD are similar. On the other hand, in case of convergence to a degenerate critical point, then the performance of the two algorithms can be sharply different. Below are some experimental results illustrating that both scenarios do happen in reality. 

The setups are as follows. We choose $\alpha =0.5$ for Armijo's condition. 

For the usual Backtracking GD, we choose $\beta =0.7$ and $\delta _0=1$. 

For Unbounded Backtracking GD: we choose $\beta =0.7$ and $\delta _0=1$ as in the usual Backtracking GD. We choose the function $h(t)=\delta _0$ if $t>1$, and $h(t)=\delta _0/\sqrt{t}$ if $t\leq 1$. For the readers' convenience, we recall here the update rule for Unbounded Backtracking GD \cite{truong}: At step $n$, we start with $\delta =\delta _0$. If $\delta $ does not satisfy Armijo's condition, then we reduce $\delta $ by $\delta \beta$ until it satisfies Armijo's condition, hence in this case we proceed as in the usual Backtracking GD. On the other hand, if $\delta $ does satisfy Armijo's condition, then we increase it by $\delta /\beta $ while both Armijo's condition and $\delta \leq h(||\nabla f(x_n)||)$ is satisfied. We choose $\delta _n$ to be the final value of $\delta$, and update $x_{n+1}=x_n-\delta _n\nabla f(x_n)$.   

We will stop when either the iterate number is $10^6$ or when the gradient of the point is $\leq 10^{-10}$. 

 {\bf Example 1:}  We look at the function $f(x,y)=x^3 sin(1/x) +y^3sin(1/y)$ and start from the initial point $z_0=(4,-5)$. After 10 steps, both algorithms Backtracking GD and Unbounded Backtracking GD arrive at the same point $(0.09325947,-0.09325947)$ which is very close to a non-degenerate local minimum of the function. 
 
 {\bf Example 2:} We look at the function $f(x,y)=x^4+y^4$ and start from the initial point $z_0=(0.1,15)$. This function has a degenerate global minimum at $(0,0)$. After $10^6$ steps, Backtracking GD arrives at the point $(0.00111797,0.00111802)$ with learning rate $1$. On the other hand, only after 89 steps, Unbounded Backtracking GD already arrives at a better point $(0.00025327, 0.00025327)$ with learning rate $90544.63441298596$ much bigger than $1$. 
 
Finally, we present a heuristic argument showing that Armijo's condition and backtracking manner of choosing learning rates could prevent a pathological scenario not covered by the convergence result in \cite{truong}. More precisely, we use the following update rule: it is like the update rule for the discrete version of Unbounded Backtracking GD mentioned above, except that we do not constrain $\delta $ by any function $h(||\nabla f(z_n)||)$. The pathological scenario is that the constructed sequence $\{z_n\}$ contains both a bounded and an unbounded subsequence, and the bounded subsequence converges to a critical point $z_{\infty}$. Since as mentioned, modifications of Backtracking GD in \cite{truong2, truong3} can avoid saddle points, we expect that with the above update rule the sequence $\{z_n\}$ can also avoid saddle points. Then the point $z_{\infty}$ is expected to be a local minimum. There is expected a small open neighbourhood $U$ of $z_{\infty}$ for which $\min _{z\in \partial U}f(z)>f(z_{\infty})$. Now, the backtracking manner of choosing learning rates is expected to have this effect: if $z\in U$ is very close to $z_{\infty}$, then the choice of $\delta (z)$ - since at most will be increased by $\beta$ at a time and must keep the value of the function not increased - will not be enough to allow the resulting point $z-\delta (z)\nabla f(z)$ to escape $U$. (Since $||\nabla f(z_n)||$ is very small, it is expected that if $\delta '$ is the largest positive number so that $z_n-\delta '\nabla f(z_n)$ stays in $U$, then the next value $\delta '/\beta$ is expected to make $z_n-\delta '\nabla f(z_n)/\beta $ stay close to $\partial U$, which will force $f(z_n-\delta '\nabla f(z_n)/\beta )>f(z_n)$ - a condition prohibited by Armijo's condition.)  Therefore, we expect that if there is a sequence $\{z_{n_j}\}$ converging to $z_{\infty}$, then the whole sequence $\{z_n\}$ must be bounded, and the above pathological scenario cannot happen. It would be good if the above heuristic argument can be realised at least for $C^2$ cost functions.

\subsection{Backtracking GD has correct units}

In \cite{zeiler} where he introduced Adadelta, Zeiler has an interesting interpretation of whether a numerical method is "right" or not, based on the idea of "correct unit". Here we show that Backtracking GD has the correct unit, thus gives more support to why it is effective. The argument is of course  non-rigorous, but we hope that this explanation can be amusing and can encourage more interest in using Backtracking GD in practical applications, in particular in Deep Learning. 

The idea is as follows. If we have an equality $LHS=RHS$, then whenever the LHS has a certain unit, then so is the RHS. For example, in the formula for velocity $v=x/t$, if the unit of $x$ is m and the unit of $t$ is s, then the unit of $v$ must be $m/s$.  Likewise, in numerical methods, if we define $x_{n+1}=x_n+\xi _n$, then the unit of $\xi _n$ must be equal that of $x_n$ and  $x_{n+1}$. 

To make the presentation simple, we will choose dimension $k=1$, and hence our map $f$ is from $\mathbb{R}$ to $\mathbb{R}$. In this case, we can write $f'(x)$ for $\nabla f(x)$. For an object $z$, we write $Unit(z)$ for its unit. Our convenience is that if a constant $\alpha$ is not bound in any relation (equality, in equality and so on), then  it is unitless.

By definition
\begin{eqnarray*} 
f'(x)=\frac{\Delta f}{\Delta x},
\end{eqnarray*}
where $\Delta$ is the difference, and for any object $z$ we have $Unit(\Delta z)=Unit (z)$. Therefore, we obtain $Unit(f')=Unit(f)/Unit(x)$. 

Similarly, $f"(x)=\Delta f'/\Delta x$ implies that $Unit(f")=Unit(f')/Unit(x)=Unit(f)/Unit(x)^2$. 

Zeiler analysed the unit correctness of some common gradient descent methods appearing before Adadelta: Standard GD, Momentum, Adagrad and Newton's method. Here we repeat the analysis for Standard GD and Newton's. 

For Standard GD, the update rule is $x_{n+1}=x_n-\delta _0f'(x_n)$. Since $\delta$ is an unbound constant, we have that $\delta _0$ is unitless. Hence, we have a mismatch because $Unit(x)^2\not= Unit(f')$ in general. This can be interpreted in that Standard GD is not the "right" method for a general $C^1$ function. Similarly, Zeiler showed that Momentum and Adagrad do not have correct units. 

For Newton's method, the update rule is $x_{n+1}-x_n=-f'(x_n)/f"(x_n)$. Here we have unit correctness because the unit of RHS is
\begin{eqnarray*}
Unit(f'/f")=[Unit(f)/Unit(x)]/[Unit(f)/Unit(x)^2]=Unit(x), 
\end{eqnarray*}
which is the same as that of LHS. One weak point of Newton's method is however that it is not guaranteed to be a descent method, that is there is no guarantee that $f(x_{n+1})\leq f(x_n)$ for all $n$. 
 Zeiler designed his algorithm Adadelta as a way to make Adagrad have correct unit. However, again this method is not guaranteed to be descent. 
 
Now we show that Backtracking GD has correct unit. In deed, we choose $\delta (x_n)$ as the largest $\delta$ among $\{\beta ^n\delta _0: ~n=0,1,2,\ldots\}$ so that Armijo's condition 
\begin{eqnarray*}
f(x-\delta f'(x))-f(x)\leq -\alpha \delta |f'(x)|^2.
\end{eqnarray*}  
 Since $x-\delta f'(x)$ appears as an argument for the function $f$, we must have $Unit(\delta f'(x))=Unit(x)$, which implies that
 \begin{eqnarray*}
 Unit(\delta )=Unit(x)/Unit( f'(x))=Unit(x)/[Unit(f)/Unit(x)]=Unit(x)^2/Unit(f). 
 \end{eqnarray*} 
For Armijo's condition to have correct unit, the necessary and sufficient condition is then that $\alpha$ is unitless. Likely, we check that $\beta$ is unitless, and $Unit(\delta _0)=Unit(\delta (x))=Unit(x)^2/Unit(f)$.

Likewise, we can now check that in case $\nabla f$ is Lipschitz continuous with Lipschitz constant $L$, then the Standard GD update with learning rate $\delta _0~1/L$ has correct unit. To see this, we first observe that since the constant $L$ is bound in the inequality $|f'(x)-f'(y)|\leq L|x-y|$, it follows that 
\begin{eqnarray*}
Unit(L)=Unit(f')/Unit(x)=Unit(f)/Unit(x)^2,
\end{eqnarray*}
and hence the update rule $\Delta x_n=-\delta _0f'(x_n)$ has correct unit. We can see this fact also by observing that in this case the Standard GD is a special case of the Backtracking GD, and hence also has correct unit. For example, if we choose the learning rate to be too much bigger than $1/L$, then the sequence may diverge to $\infty$, that is the update rule is not "right". If we instead choose the learning rate to too much smaller than $1/L$, then convergence can be guaranteed but the limit point may not be a critical point of $f$. 

On the other hand, for Diminishing GD, where we pre-choose a sequence $\delta _n$ so that $\lim _{n\rightarrow\infty}\delta _n=0$ and $\sum _{n}\delta _n=\infty$, independent of functions $f$, then it is clear that $\delta _n$'s are unitless. Then the update rule for Diminishing GD does not have "correct unit".

\subsection{Acknowledgments} We thank anonymous comments for inspiring our study in Section 0.1. The first author is  supported by Young Research Talents grant 300814 from Research Council of Norway.

\end{document}